\documentclass[11pt, reqno]{amsart}

\usepackage{amsthm,amsmath,amssymb,xcolor}

\usepackage[margin=1in]{geometry}

\usepackage[utf8]{inputenc}
\usepackage[T1]{fontenc}

\usepackage[hidelinks]{hyperref}
\usepackage{url}

\usepackage[noabbrev,capitalize]{cleveref}

\usepackage[color,final]{showkeys} 

\colorlet{refkey}{orange!20}
\colorlet{labelkey}{blue!30}

\newtheorem{theorem}{Theorem}

\newtheorem{lemma}[theorem]{Lemma}

\newtheorem{conjecture}[theorem]{Conjecture}
\newtheorem*{question*}{Question}

\theoremstyle{definition}

\newtheorem*{definition*}{Definition}

\theoremstyle{remark}
\newtheorem*{remark}{Remark}

\newcommand{\abs}[1]{\left\lvert#1\right\rvert}
\newcommand{\norm}[1]{\left\lVert#1\right\rVert}

\newcommand{\RR}{\mathbb{R}}

\title[Counterexample to the Bollob\'as--Riordan conjectures]
{A counterexample to the Bollob\'as--Riordan conjectures \\ on sparse graph limits}

\author[Sah]{Ashwin Sah}
\author[Sawhney]{Mehtaab Sawhney}
\author[Tidor]{Jonathan Tidor}
\author[Zhao]{Yufei Zhao}
\address{Massachusetts Institute of Technology, Cambridge, MA 02139, USA}
\email{\{asah,msawhney,jtidor,yufeiz\}@mit.edu}
\date{}

\thanks{
Tidor was supported by NSF Graduate Research Fellowship Program DGE-1122374. 
Zhao was supported by NSF Award DMS-1764176, the MIT Solomon Buchsbaum Fund, and a Sloan Research Fellowship.
}

\begin{document}

\begin{abstract}
Bollob\'as and Riordan, in their paper ``Metrics for sparse graphs,'' proposed a number of provocative conjectures extending central results of quasirandom graphs and graph limits to sparse graphs.
We refute these conjectures by exhibiting a sequence of graphs with convergent normalized subgraph densities (and pseudorandom $C_4$-counts), but with no limit expressible as a kernel.
\end{abstract}

\maketitle
The study of pseudorandom and quasirandom graphs, initiated by Thomason~\cite{Tho87,Tho87a} and Chung, Graham, and Wilson~\cite{CGW89}, plays a central role in graph theory. 
A particularly nice aspect of this theory is that many notions of quasirandomness are equivalent for dense graph sequences.
The theory of graph limits~\cite{Lov12}, developed by Lov\'asz and collaborators, further generalizes these concepts.
Some of the central results of these theories are summarized below. 
Consider a sequence of graphs $G_n$ whose number of vertices goes to infinity with $n$.
We write $|G|$ and $e_G$ respectively for the number of vertices and edges of $G$, 
and $t(F, G) = \hom(F, G) |G|^{-|F|}$ for the homomorphism density of $F$ in $G$.
\begin{enumerate}
    \item \emph{$C_4$ counts control quasirandomness \cite{CGW89}.}  If $t(K_2, G_n) \to p$ and $t(C_4, G_n) \to p^4$ for some constant $p$, then $t(F, G_n) \to p^{e_F}$ for all graphs $F$, 
    	and furthermore $G_n$ converges to $p$ in the cut norm (i.e., satisfies a discrepancy condition).
    \item \emph{Existence of graph limits \cite{CGW89}.} If $t(F, G_n)$ converges as $n \to \infty$ for every $F$, then there exists a graphon $W \colon [0,1]^2 \to [0,1]$ such that $t(F, G_n) \to t(F, W)$.
    \item \emph{Equivalence of convergence \cite{BCLSV08}.}  $t(F, G_n)$ converges as $n \to \infty$ for every $F$ if and only if $G_n$ is a Cauchy sequence with respect to the cut metric.
\end{enumerate}

Implications concerning subgraph densities often fail for naive generalizations to sparse graphs. Here we call a sequence of graphs $G_n$ \emph{sparse} if $e_{G_n} / |G_n|^2 \to 0$ as $n \to \infty$. We normalize all the quantities considered according to the decaying edge-density.

There is much interest in extending the above ideas to sparse graphs. 
The first such systematic study was undertaken by Bollob\'as and Riordan~\cite{BR09}.
They considered natural notions of convergence and metrics for sparse graphs,
	and gave many interesting results and examples,
	as well as a long list of provocative conjectures. 
A recurring theme in their paper, as well as in other works in this area,
	is that one quickly runs into difficulties as soon as subgraph counts are involved. 
The lack of a general purpose ``counting lemma'' in sparse graphs appears to be a fundamental difficulty.
This issue lies at the heart of the sparse regularity method of Conlon, Fox, and Zhao~\cite{CFZ14,CFZ14a,CFZ15}, who developed novel counting lemmas in sparse graphs and hypergraphs under additional pseudorandomnesses hypotheses, building on and simplifying the Green--Tao theorem on arithmetic progressions in the primes~\cite{GT08}.
Some of the subsequent extensions of the Bollob\'as--Riordan sparse graph limit theory, in particular the $L^p$ theory of sparse graph limits~\cite{BCCZ18,BCCZ19}, largely avoids the issues of subgraph counts in favor of other metrics.

Given real $p>0$ and graphs $F$ and $G$, we define the \emph{normalized $F$-density} in $G$ to be
\[
 t_p(F,G) = \frac{\hom(F,G)}{p^{e_F}|G|^{|F|}}.
\]
Here we will primarily be concerned with $N$-vertex graphs with edge density $p = N^{-o(1)}$, 
	so that there is only a lower order difference between homomorphism counts and subgraph counts 
	(after accounting for automorphisms of $H$).
The normalization in $t_p(F, G)$ is chosen so that for a sequence of random graphs $G_n = G(n,p)$, 
	one has $t_p(F, G_n) \to 1$ for all $F$ almost surely.

A \emph{kernel} is a symmetric measurable function $W \colon [0,1]^2 \to [0,\infty)$, where symmetric means that $W(x,y) = W(y,x)$. 
(The word \emph{graphon} is often used in the literature for kernels with $[0,1]$-values.)
We say that a kernel is \emph{bounded} if there is some real $C$ so that $0 \le W \le C$ holds pointwise.
Given a graph $H$, we define the \emph{$H$-density} of a kernel $W$ to be
\[
t(H, W) = \int_{[0,1]^{V(H)}} \prod_{uv \in E(H)} W(x_u, x_v) \, \prod_{v \in V(H)} dx_v.
\]

Bollob\'as and Riordan~\cite{BR09} proposed the following conjectures. 
Throughout, let $G_n$ be a sequence of graphs 
	with edge-density $p_n = 2e_{G_n}/|G_n|^2$ 
	satisfying $p_n = |G_n|^{-o(1)}$. 
For a graph $F$, write 
\[
c_F = \lim_{n \to \infty} t_{p_n}(F, G_n).
\]
\begin{itemize}
    \item \cite[Conjecture~3.4]{BR09} \label{conj:3.4}
    If $c_F$ exists and is finite for all graphs $F$, then there is some kernel $W$ such that $t(F, W) = c_F$ for all graphs $F$.
    \item \cite[Conjecture~3.3]{BR09} \label{conj:3.3}
    If $c_F$ exists for all graphs $F$ and $\sup_F c_F^{1/e_F} < \infty$, then there is a bounded kernel $W$ such that $t(F, W) = c_F$ for all graphs $F$.
    \item \cite[Conjecture~3.21]{BR09} \label{conj:3.21}
    If $c_F$ exists and is finite for all graphs $F$ and $c_{C_4} = 1$, then $c_{K_3} = 1$.
    \item \cite[Conjecture~3.9]{BR09} \label{conj:3.9}
    If $c_F$ exists and is finite for all graphs $F$ and $c_{C_4} = 1$, then $c_F = 1$ for all graphs $F$.
\end{itemize}

There are additional conjectures in \cite{BR09} that we do not state here precisely since they require additional definitions. 
In particular, Conjecture~3.22 concerns graphs of sparser densities and would imply Conjecture~3.21.
Conjecture~5.5 would imply Conjecture~3.3. 
Conjectures~5.6 and 5.7 propose equivalences between convergence of subgraph densities and convergence in cut metric, and they would imply Conjecture~5.5.

We provide a single counterexample that refutes all conjectures in \cite{BR09}.

\begin{theorem}
There exists a sequence of graphs $G_n$ with $|G_n| \to \infty$ and edge density $p_n = |G_n|^{-o(1)}$
such that for every graph $F$, writing $\triangle_F$ for the number of triangles in $F$,
\[
t_{p_n} (F, G_n) \to e^{-\triangle_F} \qquad \emph{as } n \to \infty.
\]
Moreover, there is no kernel $W$ satisfying $t(F, W) = e^{-\triangle_F}$ for all graphs $F$.
\end{theorem}

\begin{proof}
Let $G = G_n = K_n^{\otimes n^2}$, the $n^2$-th tensor power of $K_n$. Explicitly, this graph has vertex set $[n]^{n^2}$, with two tuples adjacent precisely when they differ in every coordinate.
Its edge density is $p = p_n = (1-n^{-1})^{n^2} = (e^{-1/2}+o(1))e^{-n}$.
Note that $\hom(F, K_n)$ counts proper $n$-colorings of $F$. It is a standard result in graph theory (easily proved using inclusion-exclusion) that
\[
\hom(F, K_n) 
= 
n^{|F|} - e_F n^{|F|-1} + \left(\binom{e_F}{2}-\triangle_F \right) n^{|F|-2}+O_F(n^{|F|-3}).
\]
Since $\hom(F, K_n^{\otimes n^2}) = \hom(F, K_n)^{n^2}$, 
\begin{align*}
t_p(F, G) &= p^{-e_F} |G|^{-|F|}  \hom (F, K_n)^{n^2} \\
&=  (1-n^{-1})^{-e_Fn^2} 
 \left(1-e_F n^{-1} + \left(\binom{e_F}{2}  -\triangle_F\right) n^{-2}+O_F(n^{-3})\right)^{n^2}\\
& = (1-n^{-1})^{-e_Fn^2} (1-n^{-1})^{e_Fn^2}\left(1-\triangle_Fn^{-2}+O_F(n^{-3})\right)^{n^2}\\
&= \left(1- \triangle_F n^{-2} +O_F(n^{-3}) \right)^{n^2} \\
&\to e^{-\triangle_F} \qquad \text{as } n \to \infty.
\end{align*}

Finally, the standard proof of the equivalence of quasirandomness for dense graphs shows that if a kernel $W$ satisfies $t(K_2, W) = t(C_4, W) = 1$, then $W = 1$ almost everywhere (see \cite[Claim 11.63]{Lov12}, whose proof does not require $W$ to be bounded), and hence $t(F, W) = 1$ for all graphs $F$. Thus there is no $W$ satisfying $t(F, W) = e^{-\triangle_F}$ for all graphs $F$.
\end{proof}

\begin{remark}
After normalizing by dividing by the edge density, the kernels corresponding
to $G_n$ converge in cut norm to the constant kernel.
This is a result of the following lemma applied with $W_n$ being the associated graphon of $G_n$ divided by $p_n$. 
As a consequence (see \cite[Lemma 4.2]{BR09}), the graph sequence satisfies the bounded density assumption \cite[Assumption 4.1]{BR09} (also known under the names ``no dense spots''~\cite{Koh97}  and ``$L^\infty$ upper regular''~\cite{BCCZ18,BCCZ19}).

One can obtain a sequence of graphs with similar properties and $|G_n| = n$ by slowly blowing-up the above construction (see \cite[Remark 3.14]{BR09}).
\end{remark}

Recall the cut norm of $U \colon [0,1]^2 \to \RR$ is defined by $\norm{U}_\square = \sup_{A,B \subseteq [0,1]} \abs{\int_{A \times B} U}$.

\begin{lemma}
If a sequence $W_n$ of kernels satisfies $t(F, W_n) \to 1$ whenever $F$ is a subgraph of $C_4$, then $\norm{W_n - 1}_\square \to 0$. 
\end{lemma}

\begin{proof}
Applying Cauchy--Schwarz twice (e.g., \cite[Lemma 8.12]{Lov12}) and expanding,
\begin{align*}
\norm{W_n - 1}_\square^4 
&\le t(C_4, W_n - 1)\\
&= t(C_4, W_n) - 4 t(P_3, W_n) + 4 t(K_{2,1}, W_n) + 2 t(K_2, W_n)^2 - 4 t(K_2, W_n) + 1\\
&\to 0. \qedhere
\end{align*}
\end{proof}

Our counterexample illustrates a fundamental difficulty with counting in sparse graphs, and suggests that additional hypotheses, such as those in \cite{CFZ14,CFZ15}, may indeed be necessary. 

We close by offering a new conjecture. We say that a set $\mathcal{S}$ of graphs is \emph{sparse forcing} if given a 
sequence of graphs $G_n$ with $|G_n| \to \infty$ and edge density $p_n = |G_n|^{-o(1)}$ such that the limit $c_F = \lim_{n\to\infty} t_{p_n} (F, G_n)$ exists for every graph $F$ and satisfies $\sup_F c_F^{1/e_F} < \infty$, and provided $c_F = 1$ for all $F \in \mathcal{S}$, one necessarily has $c_F = 1$ for all graphs $F$. In other words, having quasirandom density of graphs in $\mathcal{S}$ forces quasirandom density of all graphs. Our counterexample above shows that no set of triangle-free graphs can be sparse forcing. 
On the other hand, in the dense setting, i.e., for constant $p_n$, $\{K_2, C_4\}$ is forcing, and a well-known conjecture~\cite{CGW89} says that $\{K_2, H\}$ is forcing whenever $H$ is a bipartite graph with at least one cycle.

\begin{conjecture}
No finite set of graphs $\mathcal{S}$ can be sparse forcing.
\end{conjecture}

\subsection*{Acknowledgments}
We thank the anonymous referee for helpful comments regarding the presentation.


\end{document}